\documentclass[10pt]{amsart}

\usepackage[centertags]{amsmath}
\usepackage{amsfonts}
\usepackage{amsthm}
\usepackage{newlfont}
\usepackage{amscd}
\usepackage{amsgen}
\usepackage{amssymb}

\newlength{\defbaselineskip} 

 \theoremstyle{plain} \newtheorem{thm}{Theorem}[section]
 \newtheorem{lemm}[thm]{Lemma}
\newtheorem{prop}[thm]{Proposition}
\newtheorem{obs}[thm]{Observation}
\theoremstyle{definition}

\newtheorem{rem}[thm]{Remark}

 \numberwithin{equation}{section}
\numberwithin{equation}{section}
\theoremstyle{definition}

\title{Some degenerations of $G_2$ and Calabi-Yau varieties}
\author{Micha{\l}\ Kapustka}

\begin{document}

\begin{abstract}
We introduce a variety $\hat{G}_2$ parameterizing isotropic five-spaces of a
general
degenerate four-form in a seven dimensional vector space. It is in a natural way
a degeneration of the variety $G_2$, the adjoint variety of the simple Lie group
$\mathbb{G}_2$.
It occurs that it is also the image of $\mathbb{P}^5$ by a system of
quadrics containing a twisted cubic. Degenerations of this twisted cubic to
three lines give rise to degenerations of $G_2$ which are toric Gorenstein
Fano fivefolds.
We use these two degenerations to construct geometric transitions between
Calabi--Yau threefolds.
We prove moreover that every polarized K3 surface of Picard number 2, genus 10, and admitting a $g^1_5$ 
appears as linear sections of the variety $\hat{G}_2$.
\end{abstract}
\maketitle
\section{Introduction}
We shall denote by $G_2$ the adjoint variety of the complex simple Lie group
$\mathbb{G}_2$.
In geometric terms this is a subvariety of the Grassmannian $G(5,V)$ consisting
of 5-spaces isotropic with respect to a chosen non-degenerate 4-form $\omega$ on
a 7-dimensional vector space $V$.
In this context the word non-degenerate stands for 4-forms contained in the open
orbit in $\bigwedge^4 V$ of the natural
action of $\operatorname{Gl}(7)$. It is known (see \cite{ott}) that this open orbit is the complement of a
hypersurface of degree 7. The hypersurface
is the closure of the set of 4-forms which can be decomposed into the sum of
3 simple forms. The expected number
of simple forms needed to decompose a general 4-form is also 3, meaning that our
case is defective. In fact this is the only
known example (together with the dual $(k,n)=(3,7)$) with $3 \leq k\leq n-3$ in which a
general $k$-form in an n-dimensional space cannot be decomposed
into the sum of an expected number of simple forms.
A natural question comes to mind. What is the variety $\hat{G}_2$ of 5-spaces
isotropic with respect to a generic 4-form from the hypersurface of degree 7?
>From the above point of view it is a variety which is not expected to exist.
We prove that the Pl\"ucker embedding of $\hat{G}_2$ is linearly isomorphic to the
closure of the image of $\mathbb{P}^5$ by the map defined by quadrics containing a fixed twisted cubic.
We check also that $\hat{G}_2$ is singular along a plane and appears as a flat
deformation of $G_2$.

Next, we study varieties obtained by degenerating the twisted cubic to a
reducible cubic. All of them appear to be flat deformations of $G_2$.
However only one of them appears to be a linear section of $G(5,V)$. It
corresponds to the variety of 5-spaces isotropic with respect to a 4-form from
the tangential variety to the Grassmannian $G(4,7)$. The two other degenerations
corresponding to configurations of lines give rise to toric degenerations of
$G_2$.
The variety $G_2$ as a spherical variety is proved in \cite{Br} to have such
degenerations, but for $G_2$ it is not clear whether the constructed degeneration is Fano.
In the context of applications, mainly for purposes of mirror symmetry of
Calabi--Yau manifolds, it is important that these degenerations lead to
Gorenstein toric Fano varieties. Our two toric varieties  are both Gorenstein and Fano,
they admit respectively 3 and 4 singular strata of codimension 3 and degree 1.
Hence the varieties obtained by intersecting these toric 5-folds with a quadric
and a hyperplane have 6 and 8 nodes respectively. The small resolutions of these
nodes are Calabi--Yau threefolds which are complete intersections in smooth
toric Fano 5-folds and are connected by conifold transition to the Borcea Calabi--Yau
threefolds of degree 36, which are sections of $G_2$ by a hyperplane and a
quadric, and will be denoted $X_{36}$.  This is the setting for the methods developed in
\cite{BCKS} to work and provide a partially conjectural construction of mirror. 
Note that in \cite{BK} the authors found a Gorenstein
toric Fano fourfold whose hyperplane
section is a nodal Calabi--Yau threefold admitting a smoothing which has the
same hodge numbers, degree, and degree of the second Chern class as $X_{36}$. It follows by a theorem of Wall that it
is diffeomorphic to it and by connectedness of the Hilbert scheme is also a flat deformation of it. However, a priori the two varieties can be in
different components of the Hilbert scheme, hence do not give rise to a properly
understood conifold transition. In this case it is not clear what is the
connection between the mirrors of these varieties.

The geometric properties of $\hat{G}_2$ are also used in the paper for the
construction of another type of geometric transitions. A pair of geometric
transitions joining  $X_{36}$ and the complete intersection of a quadric and a quartic in
$\mathbb{P}^5$. The first is a conifold transition involving a small contraction of two nodes
the second a geometric transition involving a primitive contraction of type III.

In the last section we consider a different application of the considered
constructions. We apply it to the study of polarized K3 surfaces of genus 10.
By the Mukai linear section theorem (see \cite{Muk1}) we know that a generic polarized K3
surface of genus 10 appears as a complete linear section of $G_2$. A
classification of the non-general cases has been presented in \cite{JK}.
The classification is however made using descriptions in scrolls,
which is not completely precise in a few special cases. We use our construction
to clarify one special case in this classification. This is the case of polarized  K3 surfaces
$(S,L)$ of genus 10 having a $g^1_5$ (i.e. a smooth representative of $L$ admits a
$g^1_5$). In particular we prove that a smooth linear section of $G_2$ does not admit a $g^1_5$. Then, we prove that
each smooth two dimensional linear section of $\hat{G}_2$ has a $g^1_5$ and that K3
surfaces appearing in this way form a component of the moduli space of such surfaces.
More precisely we get the following. 
\begin{prop} \label{one g15}
Let $(S,L)$ be a polarized K3 surface of genus 10 such that $L$ admits exactly one $g^1_5$, 
then $(S,L)$ is a proper linear section of one of the four considered degenerations of $G_2$.
\end{prop}

\begin{prop}\label{Pic 2}
 If $(S,L)$ is a polarized K3 surface of genus 10 such that $L$ admits a
$g^1_5$ induced by an elliptic curve, and $S$ has Picard number 2, then $(S,L)$ is a proper linear
 section of $\hat{G}_2$.
\end{prop}
The methods used throughout the paper are elementary and rely highly on direct
computations in chosen coordinates including the use of Macaulay2 and Magma.
\section{The variety $G_2$}
In this section we recall a basic description of the variety $G_2$ using
equations.

\begin{lemm}
The variety $G_2$ appears as a five dimensional section of the Grassmannian
$G(2,7)$
with seven hyperplanes (non complete intersection). It parametrizes the set of
2-forms
$\{\left[v_1\wedge v_2\right] \in G(2,V)\ |\  v_1\wedge v_2\wedge \omega=0\in
\bigwedge^6V\}$,
 where $V$ is a seven dimensional vector space and $\omega$ a non-degenerate
four-form on it.
\end{lemm}

By \cite{ott} we can choose $\omega= x_1\wedge x_2 \wedge x_3 \wedge
x_7+x_4\wedge x_5
\wedge x_6 \wedge x_7 + x_2\wedge x_3 \wedge x_5 \wedge x_6 + x_1\wedge x_3
\wedge x_4
\wedge x_6+x_1\wedge x_2 \wedge x_4 \wedge x_5$. The variety $G_2$ is then
described in its linear span $W\subset\mathbb{P}(\bigwedge^2V)$ with coordinates $(a\dots n)$  by
$4\times 4$ Pfaffians of the matrix:

$$\left(\begin{array}{ccccccc}
0&-f&e&g&h&i&a\\
f&0&-d&j&k&l&b\\
-e&d&0&m&n&-g-k&c\\
-g&-j&-m&0&c&-b&d\\
-h&-k&-n&-c&0&a&e\\
-i&-l&g+k&b&-a&0&f\\
-a&-b&-c&-d&-e&-f&0
  \end{array}\right).
$$

\section{The variety $\hat{G}_2$}

>From \cite{ott} there is a hypersurface of degree 7 in $\bigwedge^4 V$
parameterizing four-forms which may be written as a sum of three pure forms.
The generic element of this hypersurface corresponds to a generic degenerate
four-form $\omega_0$. After a suitable change of coordinates we may assume (see \cite{ott})
that $\omega_0= x_1\wedge x_2 \wedge x_3 \wedge x_7+x_4\wedge x_5 \wedge x_6
\wedge x_7 + x_2\wedge x_3 \wedge x_5 \wedge x_6 + x_1\wedge x_3 \wedge x_4
\wedge x_6$. Let us consider the variety $\hat{G}_2=\{\left[v_1\wedge v_2\right]
\in G(2,V) \ |\  v_1\wedge v_2\wedge \omega_0=0\in \bigwedge^6V\}$.
Analogously as in the non-degenerate case it is described in it's linear span by
$4\times 4$ Pfaffians of a matrix of the form

$$\left(\begin{array}{ccccccc}
0&0&e&g&h&i&a\\
0&0&-d&j&-g&l&b\\
-e&d&0&m&n&k&c\\
-g&-j&-m&0&0&-b&d\\
-h&g&-n&0&0&a&e\\
-i&-l&k&b&-a&0&f\\
-a&-b&-c&-d&-e&-f&0
  \end{array}\right).
$$

Directly from the equations we observe that $\hat{G}_2$ contains a smooth Fano
fourfold $F$ described in the space
$(b,c,d,f,j,l,m,k)$ by the $4\times 4$ Pfaffians of the matrix
 $$\left(\begin{array}{ccccccc}
0&-d&j&l&b\\
d&0&m&k&c\\
-j&-m&0&-b&d\\
-l&-k&b&0&f\\
-b&-c&-d&-f&0
  \end{array}\right)
$$

\begin{rem}\label{fana}
 In fact we see directly a second such Fano fourfold $F'$ isomorphic to $F$ and meeting $F$ in a plane.
It is analogously the intersection of $\hat{G}_2$ with the space $(e,h,i,a,n,c,k,f)$.
We shall see that there is in fact a one parameter family of such Fano fourfolds any two
intersecting in the plane $(c,k,f)$.
\end{rem}
\begin{obs}\label{rzut na P1xP5}
The image of the projection of $\hat{G}_2$ from the plane spanned by $(c,k,f)$
is a hyperplane section of $\mathbb{P}^1\times\mathbb{P}^5$.
\end{obs}
\begin{proof} The projection maps $\hat{G}_2$ to $\mathbb{P}^{10}$ with
coordinates $(a,b,d,e,g,h,i,j,l,m,n)$
Observe that the equations of the projection involve $2\times 2$ minors of the
matrix
$$\left(\begin{array}[c]{cccccc}
 e&g&h&i&a&-n\\
-d&j&-g&l&b&m
\end{array}\right),$$
as these equations appear in the description of $\hat{G}_2$ and do not involve
$c,k,f$.
It follows that the image is contained in a hyperplane section $P$ of a
$\mathbb{P}^1\times\mathbb{P}^5$.
Next we check that the map is an isomorphism over the open subset given by $g=1$
of $P$.
\end{proof}

\begin{prop} \label{hilbP3} The Hilbert scheme of projective 3-spaces contained
in $\hat{G}_2$ is a conic.
Moreover the union of these 3-spaces is a divisor $D$ of degree $8$ in
$\hat{G}_2$.
\end{prop}

\begin{proof}
We start by proving the following lemmas.
\begin{lemm}\label{3secantplanesg27}
  Let a plane $P$ have four points of intersection with a $G(5,V)$, such that
they
  span this plane. Then $P\cap G(5,V)$ is
  a conic parameterizing all five-spaces containing a three-space $W$.
  \end{lemm}
  \begin{proof} The proof follows from \cite{Muk2}, as three points in $G(2,V)$
  always lie in a $G(2,A)$ for some subspace $A$ of dimension 6.
  \end{proof}

\begin{lemm}\label{lemhilbP3}
A projective three-space $\Pi\subset G(2,7)$ is contained in $\hat{G}$ if and
only if there exists a vector $u$ in $V$ and a four-space
$v_1\wedge v_2\wedge v_3\wedge v_4 \in G(4,7)$ such that $u\wedge
\omega_0=u\wedge v_1\wedge v_2\wedge v_3\wedge v_4 $ and $\Pi$ is generated
by $u\wedge v_1$, $u\wedge v_2$, $u\wedge v_3$, $u\wedge v_4$.
\end{lemm}
\begin{proof} To prove the if part we observe that our
conditions imply $u\wedge v_i\wedge \omega_0=0$ for $i=1,\dots,4$.
Let us pass to the proof of the only if part.
Observe first that any projective three-space contained in $G(2,V)$ is spanned
by four points of the form
$u\wedge v_1$,$u\wedge v_2$,$u\wedge v_3$,$u\wedge v_4$. By our assumption on
$\omega_0$ the form $u\wedge \omega_0 \neq 0$,
and it is killed by the vectors $u,v_1,\dots,v_4$, hence equals $u\wedge
v_1\wedge v_2\wedge v_3\wedge v_4$.
\end{proof}
Now, it follows from Lemma \ref{lemhilbP3} that the set of projective three-spaces
contained in
$\hat{G}$ is parametrized by those $\left[v\right]\in\mathbb{P}(V)$
for which $v\wedge \omega_0 \in G(5,7)$. The form $\omega_0$ may be written as
the sum of three simple forms corresponding to three subspaces
$P_1$, $P_2$, $P_3$ of dimension 4 in $V$, each two meeting in a line and no
three having a nontrivial intersection. Hence the form
$v\wedge \omega_0$ may be written as the sum of three simple forms corresponding
to three subspaces of dimension 5 each spanned by $v$ and one
of the spaces $P_i$. By lemma \ref{3secantplanesg27} the sum of these three
5-forms may be a simple form only if they all contain a common 3-space.
But this may happen only if $v$ lies in the space spanned by the lines $P_i\cap
P_j$. Now it is enough to see that the condition $v\wedge \omega_0$
is simple, corresponds for the chosen coordinate space to $((v\wedge
\omega_0)^*)^2=0$ and perform a straightforward computation to see that it
induces a quadratic equation on the coefficients of $v\in \operatorname{span}\{P_1\cap
P_2,\ P_1\cap P_3,\ P_2\cap P_3\}$.

In coordinates the constructed divisor is the intersection of $\hat{G}_2$ with
$\{g=h=j=0\}$. The latter defines on $G(5,7)$ the set of lines intersecting the distinguished plane.
We compute in Macaulay2 its degree.
\end{proof}

\begin{rem}
>From the above proof it follows that the form $\omega_0$ defines
a conic $Q$ in $\mathbb{P}(V)$ by  $Q=\{ \left[v\right]\in\mathbb{P}(V)\ |\  v\wedge
\omega_0 \in G(5,7)\}$.
Observe that any secant line of this conic is an element of $\hat{G}_2$. Indeed
let $v_1,v_2\in V$ be two vectors such that
$\left[v_1\right],\ \left[v_2\right]\in Q$. Then $v_i\wedge \omega_0$ defines a 5
space $\Pi_i\subset V$ for $i=1,2$. Consider now the product
$v_1\wedge v_2 \wedge \omega_0$. If it is not zero it defines a hyperplane in
$V$. It follows that $\dim(\Pi_1\cap \Pi_2)=4$ and
$\omega_0$ can then be written in the form $\omega_0=u_1\wedge u_2 \wedge u_3
\wedge u_4 +v_1 \wedge v_2 \wedge \alpha$.
 According to \cite{ott} this decomposition corresponds to a non general
degenerate form
$\omega_0$ giving us a contradiction.

The proof implies also that each $\mathbb{P}^3$ contained in $\hat{G}_2$ is a
$\mathbb{P}^3$ of lines passing through a chosen $v$ in the conic
and contained in the projective four-space corresponding to $v\wedge \omega_0$.

\end{rem}

\begin{rem}
For any three points $v_1,v_2,v_3$ lying on the distinguished conic $Q$, there
exists a decomposition of $\omega_0$ into the sum of 3 simple forms
$\alpha_1,\alpha_2, \alpha_3$ such that
$v_1\wedge(\omega_0-\alpha_1)=v_3\wedge(\omega_0-\alpha_2)=v_3\wedge(\omega_0-
\alpha_3)=0$. In other words
for any triple of points on the conic there is a decomposition with corresponding 4-spaces
$P_1$,$P_2$,$P_3$ such that $(v_1,v_2,v_3)=(P_1\cap P_2,P_1\cap P_3,P_2\cap
P_3)$.
\end{rem}

\begin{rem}
A three form defining $\hat{G}_2$ has a five dimensional family of presentations
into the sum of three simple forms corresponding to three
subspaces $P_1$, $P_2$, $P_3$, however all these presentations induce the same
space $\operatorname{span}\{P_1\cap P_2,P_1\cap P_3,P_2\cap P_3\}$.
This space corresponds to the only projective plane in $\hat{G}$ consisting of
lines contained in a projective plane. All other planes contained
in $\hat{G}$ consist of lines passing through a point and contained in a
three-space.
\end{rem}

\begin{prop} \label{maptoP5} The projection of $\hat{G}_2$ from $F$ is a
birational map onto $\mathbb{P}^5$ whose inverse is the map $\varphi$ defined
by the system of quadrics in $\mathbb{P}^5$ containing a common twisted cubic.
\end{prop}
\begin{proof} Observe that the considered projection from $F$ decomposes into a
projection from the plane spanned by $c,k,f$
and the canonical projection from $\mathbb{P}^1\times \mathbb{P}^5$ onto
$\mathbb{P}^5$. The latter restricted to $P$ is the blow down of
$\mathbb{P}^1\times\mathbb{P}^3$. It follows that the map is an isomorphism
between the open set given by $g=1$ and its image in $\mathbb{P}^5$.
Let us write down explicitly the inverse map. Let $(x,y,z,t,u,v)$ be a
coordinate system in $\mathbb{P}^5$. Consider a twisted cubic curve
given by $u=0$, $v=0$ and the minors of the matrix
$$\left(\begin{array}[c]{ccc}
x&y &z\\
t&x&y
\end{array}\right).
$$
Let $L$ be the system of quadrics containing the twisted cubic.
Choose the coordinates $(a,\dots,n)$ of $H^0(L)$ in the following way: $(a,\dots,n)=(uy,vy,yt-x^2,-vx,ux, y^2-xz,uv,
-u^2,uz, v^2,-x y+z t,vz,vt,ut).$
We easily check that the corresponding map is well defined and inverse to the projection by writing down
the matrix defining $\hat{G}_2$ with substituted coordinates.
$$\left(\begin{array}[c]{ccccccc}
   0&0& ux & u v & -u^2 & uz & uy\\
0&0& vx & v^2 & -uv & vz & vy\\
-ux&-vx&0 &v t&-u t& -x y+z t&yt-x^2\\
-u v &-v^2 &-v t&0&0&-v y &-vx\\
u^2&uv &u t&0&0&uy &u x\\
-uz &-vz &x y-z t&v y &-uy &0& y^2-xz\\
-uy&-vy&-yt+x^2&vx&-u x& -y^2+xz&0\\
  \end{array}
 \right)$$
\end{proof}
\begin{rem} \label{pencilfano}The images of the 4-dimensional projective spaces containing the twisted cubic form a pencil of smooth Fano fourfolds
each two meeting in the plane which is the image of the $\mathbb{P}^3$ spanned by the twisted cubic. The statement follows from
the fact that we can change coordinates in $\mathbb{P}^5$ and hence we can assume that any two chosen Fano fourfolds obtained in this way
are $F$ and $F'$ in Remark \ref{fana}.

\end{rem}

\begin{lemm}
The singular locus of $\hat{G}_2$ is a plane.
\end{lemm}
\begin{proof}
To see that the distinguished plane is singular it is enough to observe that
each line secant to the distinguished conic $C$ is the
common element of two projective three-spaces contained in $\hat{G}_2$. These are the spaces of lines corresponding to the
points of intersections of the secant line with $C$. By the same argument it follows also that
the divisor $D'$ is singular in the plane. To check smoothness outside let us perform the following argument.
Clearly the system $|2H-E|$ on the blow up of $\mathbb{P}^5$
in the twisted cubic separates points and tangent directions outside the pre-image transform
of the $\mathbb{P}^3$ spanned by the twisted cubic. It remains to study the image of the exceptional divisor, which is $D'$.
Now observe that for any $F_1$ and $F_2$ in the pencil of Fano fourfolds described in Remark \ref{pencilfano} there is a hyperplane
in $\mathbb{P}^{13}$ whose intersection with $\hat{G}_2$ decomposes in $F_1$, $F_2$ and $D'$. It follows that the singularities of $\hat{G}_2$
may occur only in the singularities of $D'$ and in the base points of the pencil. We hence need only to prove that  $D'$ is smooth outside
the distinguished plane. This follows directly from the description of the complement of the plane in $D'$ as a vector bundle over
the product of the twisted cubic with $\mathbb{P}^1$.

\end{proof}
\begin{rem} Observe that the map induced on $\mathbb{P}^5$ contracts only the secant lines of the twisted cubic to
 distinct points of the distinguished $\mathbb{P}^2$.
\end{rem}

\begin{rem}
A generic codimension 2 section of $\hat{G}_2$ by 2 hypersurfaces is nodal. We check this by taking a codimension 2 linear section and
looking at its singularity.
\end{rem}

\begin{lemm}The variety $\hat{G}$ is a flat deformation of $G$.
\end{lemm}
\begin{proof}
We observe that both varieties arise as linear sections of $G(2,V)$ by some
$\mathbb{P}^{10}$. Moreover we easily find
an algebraic family with those as fibers. Indeed consider the family parameterized by $t\in \mathbb{C}$ 
of varieties given in $\mathbb{P}^{13}$ by the $4\times 4$ Pfaffians of the matrices:

$$\left(\begin{array}{ccccccc}
0&-tf&e&g&h&i&a\\
f&0&-d&j&-g-tk&l&b\\
-e&d&0&m&n&k&c\\
-g&-j&-m&0&tc&-b&d\\
-h&-k&-n&-c&0&a&e\\
-i&-l&g+k&b&-a&0&f\\
-a&-b&-c&-d&-e&-f&0
  \end{array}\right).
$$
For each $t\in\mathbb{C}$ the equations describe the variety of isotropic five-spaces with respect to the form
$\omega_t= x_1\wedge x_2 \wedge x_3 \wedge
x_7+x_4\wedge x_5 \wedge x_6 \wedge x_7 + x_2\wedge x_3 \wedge x_5 \wedge x_6 + x_1\wedge x_3
\wedge x_4 \wedge x_6+t x_1\wedge x_2 \wedge x_4 \wedge x_5$. The latter is a nondegenerate fourform for $t\neq 0$. It follows 
that for $t\neq 0$ the corresponding fiber of the family is isomorphic to $G_2$ and for $t=0$ it is equal to $\hat{G}_2$. 
  
The assertion then follows from the equality of their Hilbert polynomials, which
we compute using MACAULAY 2.
\end{proof}

\section{Further degenerations}
Observe that one can further degenerate $\hat{G}_2$ by considering degenerations
of the twisted cubic $C$ in $\mathbb{P}^5$.
In particular the twisted cubic can degenerate to one of the following:
\begin{itemize}
 \item the curve $C_0$ which is the sum of a smooth conic and a line
intersecting
it in a point
\item  a chain $C_1$ of three lines spanning a $\mathbb{P}^3$
\item a curve $C_2$ consisting of three lines passing through a common point and spanning
a $\mathbb{P}^3$
\end{itemize}

Let us consider the three cases separately.

Let us start with the conic and the line. In this case we can assume that the
ideal of $C_0$
is given in $\mathbb{P}^5$ by $\{u=0,v=0\}$
and the minors of the matrix
$$\left( \begin{array}[c]{ccc}
x,y,z\\
t,x,0
\end{array} \right).
$$
Then the image of $\mathbb{P}^5$ by the system of quadrics containing $C_0$ can
also be written as a section of $G(2,7)$ consisting of
two-forms killed by the four-form  $\omega_1= x_1\wedge x_2 \wedge x_3 \wedge
x_7+ x_2\wedge x_3 \wedge x_5 \wedge x_6 + x_1\wedge x_3 \wedge x_4
\wedge x_6$. To find the deformation family we consider the family of varieties
isomorphic to $\hat{G}_2$ corresponding to twisted cubics given by
$\{u=0,v=0\}$
and the minors of the matrix
$$\left( \begin{array}[c]{ccc}
x,y,z\\
t,x,\lambda y
\end{array} \right).
$$
We conclude comparing Hilbert polynomials.
\begin{rem}
The forms $\omega_0$ and $\omega_1$ represent the only two orbits of forms in
$\bigwedge^3(V)$ whose corresponding isotropic varieties are
flat degenerations of $G_2$. To prove it we use the representatives of all 9 orbits
contained in \cite{ott} and check one by one the invariants of
varieties they define using Macaulay2. In all other cases the dimension of the isotropic variety is higher.
\end{rem}

In the case of a chain of lines the situation is a bit different.

\begin{prop} The variety $G_2$ admits a degeneration over a disc to a Gorenstein
toric Fano 5-fold whose only singularities are
3 conifold singularities in codimension $3$ toric strata of degree 1.
\end{prop}
\begin{proof}
As $\hat{G}_2$ is a degeneration of $G_2$ over a disc it is enough to prove that
$\hat{G}_2$ admits such a degeneration. We know that the latter
is the image of $\mathbb{P}^5$ by the map defined by the system of quadrics
containing a twisted cubic $C$. Let us choose a coordinate
system $(x,y,z,t,u,v)$ such that $C$ is given in $\mathbb{P}^5$ by $\{u=0,v=0\}$
and the minors of the matrix
$$\left( \begin{array}[c]{ccc}
x,y,z\\
t,x,y
\end{array} \right),
$$
then choose the chain of lines $C_1$ to be defined by $\{u=0,v=0\}$ and the
minors of the matrix
$$\left(\begin{array}[c]{ccc}
0,y,z\\
t,x,0
\end{array}\right).
$$ Let $T$ be the variety in $\mathbb{P}^{13}$ defined as the closure of the
image of $\mathbb{P}^5$ by the system of quadrics containing $C_0$.
It is an anti-canonically embedded toric variety with corresponding dual
reflexive
polytope:
$$\begin{array}{ccccccc}
(& 0 &  0 &  1 &  0 &  0 & ) \\
(& 0 &  0 &  0 &  1 &  0 &)\\
(& -1 & -1 & -1 & -1 & -1 &)\\
(& 0 &  0 &  0 &  0 &  1 &)\\
(& 1 &  0 &  0 &  0 &  0 &)\\
(& 0 &  1 &  0  &  0 &  0 &)\\
(& 1 &  1 &  1 &  0  &  1 &)\\
 (&0 &  0 & -1 &  0 & -1 &)\\
 (&1 &  1 &  0 &  1 &  1 &)
\end{array}$$

We check using Magma that the singular locus of this polytope has three conifold
singularities along codimension 3 toric strata of degree 1.

Consider the family of quadrics parameterized by $\lambda$ containing the curves
$C_{\lambda}$ defined by $\{u=0,v=0\}$ and the minors of the matrix
$$\left(\begin{array}[c]{ccc}
\lambda x,y,z\\
t,x,\lambda y
\end{array}\right).
$$
For each $\lambda\neq 0$ the equations of the image of $\mathbb{P}^5$ by the corresponding system of quadrics agree 
with the minors of the matrix:
$$\left(\begin{array}{ccccccc}
0&0&\lambda e&g&h&i&a\\
0&0&-\lambda d&j&-g&l&b\\
-\lambda e&\lambda d&0&m&n&k&c\\
-g&-j&-m&0&0&-\lambda b&d\\
-h&g&-n&0&0&\lambda a&e\\
-i&-l&k&\lambda b&-\lambda a&0&f\\
-a&-b&-c&-d&-e&-f&0
  \end{array}\right),
$$
in the coordinates $(a,\dots,n)=(uy,vy,yt-x^2,-vx,ux, y^2-xz,uv,
-u^2,uz,v^2,-x y+z t,vz,vt,ut).$
The latter define a variety isomorphic to $\hat{G_2}$ for each $\lambda$.
It is easy to check that this family degenerates to $T$ when $\lambda$ tends to
0. By comparing Hilbert polynomials we obtain that it is
a flat degeneration of $\hat{G}_2$, hence of $G_2$.
\end{proof}
In the case of the twisted cubic degenerating to three lines meeting in a point
we obtain
 a Gorenstein toric Fano 5-fold with 4 singular strata of codimension 3 and
degree 1 which is a flat deformation of $G_2$.
The corresponding dual reflexive polytope is:
 $$\begin{array}[c]{ccccccc}
    (&-1& -1& -1& -1& -1&)\\
    (&0& 0& 1& 0& 0&)\\
    (&0& 0& 0& 1& 0&)\\
    (&0& 0& 0& 0& 1&)\\
    (&1& 0& 0& 0& 0&)\\
    (&0& 1& 0& 0& 0&)\\
    (&1& 1& 1& 1& 0&)\\
    (&1& 1& 1& 0& 1&)\\
    (&1& 1& 0& 1& 1&)\\
    (&2& 2& 1& 1& 1&)\end{array}$$

\subsection{Application to mirror symmetry}
One of the methods of computing mirrors to Calabi-Yau threefolds is to find their
degenerations to complete intersections in Gorenstein toric Fano varieties.
Let us present the method, contained in \cite{BCKS}, in our context. 
We aim to use the constructed toric degeneration to compute the
mirror of the Calabi-Yau threefold $X_{36}$. As the construction is still partially conjectural we omit details in what follows.

Consider the degeneration of $G_2$ to $T$. We have, $X_{36}$ is a generic intersection 
of $G_2$ with a hyperplane and a quadric. On the other hand when we intersect $T$ with
a generic hyperplane and a generic quadric we get a Calabi-Yau threefold $\hat{Y}$ with 6 nodes. It follows
that $\hat{Y}$ is a flat degeneration of $X_{36}$. Moreover $\hat{Y}$ admits a small resolution 
of singularities, which is also a complete intersection in a toric variety. We shall denote it by $Y$.
The variety $Y$ is a smooth Calabi-Yau threefold connected to $X$ by a conifold transition. 
Due to results of \cite{BB} the variety $Y$ has a mirror family $\mathcal{Y}^*$ with generic element denoted by $Y^*$. 
The latter is found explicitly as a complete intersection in a toric variety obtained from the description 
of $Y$ by the method of nef partitions.   
The authors in \cite{BB} prove that there is in fact a canonical isomorphism between $H^{1,1}(Y)$ and $H^{1,2}(Y^*)$. 
Let us consider the one parameter subfamily $\mathcal{X}^*$ of the family $\mathcal{Y}^*$ corresponding 
to the subspace of $H^{1,2}$ consisting of elements associated by the above isomorphism 
to the pullbacks of Cartier divisors from $\hat{Y}$.

The delicate part of this mirror construction of $X_{36}$ is to prove that a generic Calabi--Yau 
threefold from the subfamily $\mathcal{X}^*$ has 6 nodes satisfying an appropriate number of relations.  
This is only a conjecture (\cite[Conj. 6.1.2]{BCKS}) which we are still unable to solve, also in this case.
Assume that the conjecture is true. We then obtain a construction of the mirror family of $X_{36}$
as a family of small resolutions of the elements of the considered subfamily.

\section{A geometric bi-transition}
In this section we construct two geometric transitions between Calabi--Yau
threefolds based on the map from Proposition \ref{maptoP5}.

Let us consider a generic section $X$ of $\hat{G_2}$ by a hyperplane and a
quadric. Observe that $X$ has exactly two nodes and admits a smoothing to a Borcea
Calabi-Yau threefold $X_{36}$ of degree 36. Observe moreover that $X$ contains a system
of
smooth K3 surfaces each two intersecting in exactly the two nodes. Namely these
are the intersections of the pencil of $F$ with the quadric and the
hyperplane. Blowing up any of them is a resolution of singularities of $X$. Let
us consider the
second resolution i.e. the one with the exceptional lines flopped. It is a
Calabi-Yau threefold $Z$ with a fibration by K3 surfaces of genus 6
and generic Picard number 1. Observe moreover that according to proposition
$\ref{maptoP5}$ the map $\varphi^{-1}$ factors through the blow up
$\tilde{\mathbb{P}}^5$ of $\mathbb{P}^5$ in the twisted cubic $C$. Let $E$ be
the exceptional divisor of the blow up and
$H$ the pullback of the hyperplane from $\mathbb{P}^5$.  In this context $Z$ is
the intersection of two generic divisors of type
$|2H-E|$ and $|4H-2E|$ respectively.
\begin{lemm}\label{PicZ}
The Picard number $\rho(Z)=2$
\end{lemm}
\begin{proof} We follow the idea of \cite{GK}.  Observe
that both systems $|2H-E|$ and $|4H-2E|$ are base point free and big on
$\tilde{\mathbb{P}}^5$. On $\tilde{\mathbb{P}}^5$ both divisors contract the
proper transform of $\mathbb{P}^3$ to $\mathbb{P}^2$. It follows by \cite[Thm.
6]{RavSrin}
that the Picard group of $Z$ is isomorphic to the Picard group of
$\tilde{\mathbb{P}}^5$ which
is of rank 2.
\end{proof}
 Moreover $Z$ contains a divisor $D'$ fibered by conics. In one hand $D'$ is the
proper transform of the divisor $D$ from Proposition \ref{hilbP3}
by the considered resolution of singularities, on the other hand $D'$ is the
intersection of $Z$ with the exceptional divisor $E$.
It follows that $D'$ is contracted to a twisted cubic in $\mathbb{P}^5$ by the
blowing down of $E$ and the contraction is primitive
by Lemma \ref{PicZ}.
It follows that $Z$ is connected by a conifold transition involving a primitive
contraction of 2 lines with $X$, and by a geometric transition involving a type
III primitive contraction with the complete intersection $Y_{2,4}\subset \mathbb{P}^5$.
\begin{rem}
We can look also from the other direction. Let $C$ be a twisted cubic, $Q_2$ a generic
quadric containing it, and $Q_4$ a generic quartic singular along it. Then the intersection
$Q_2\cap Q_4$ contains the double twisted cubic and two lines secant to it. Taking the map defined
by the system of quadrics containing $C$ the singular cubic is blown up and the two secant lines are contracted to 2 nodes.
\end{rem}

\section{Polarized K3 surfaces genus 10 with a $g^1_5$}
In this section we investigate polarized K3 surfaces of genus 10 which appear as sections of the varieties studied in this paper.
\begin{prop} \label{ciecia osobliwego g2 maja g15} A polarized K3 surface $(S,L)$, which
is a proper linear section of a $G_2$ does not admit a $g^1_5$.
\end{prop}

\begin{proof} Let us first prove the following lemma:
\begin{lemm}\label{g15 on G(2,V)} Let $p_1,\dots, p_5$ be five points on $G(2,V)$ of which no two lie on a line in $G(2,V)$ and no three
lie on a conic in $G(2,V)$ and such that they span a 3-space $P$.
Then $\{p_1,\dots, p_5\}\subset G(2,W)\subset G(2,V)$ for some five dimensional
subspace $W$ of $V$.
\end{lemm}
\begin{proof} Let $p_1,\dots,p_5$ correspond to planes $U_1,\dots,U_5\subset V$. By Lemma \ref{3secantplanesg27} we may
assume that no four of these points lie on a plane. Assume without loss of generality that $p_1,\dots,p_4$ span the 3-space.  If
$\dim(U_1+U_2+U_3+U_4)=6$ the assertion follows from \cite[Lemma 2.3]{Muk2}. We
need to exclude the case $U_1+U_2+U_3+U_4=V$. In this case (possibly changing the choice of $p_1,\dots,p_4$
from the set $\{p_1,\dots,p_5\}$) we may choose a basis in one of the two following ways
$\{v_1,\dots, v_7\}$ such that $v_1,v_2 \in U_1$, $v_3,v_4 \in U_2$, $v_5,v_6 \in
U_3$, and either $v_7,v_1+v_3+v_5 \in U_4$ or $v_7,v_1+v_3 \in U_4$. Each point of $P$ is then represented by a
bi-vector
$$w=a v_1\wedge v_2+b v_3\wedge v_4+c v_5\wedge v_6+d v_7\wedge (v_1+v_3+v_5),$$ or
$$w=a v_1\wedge v_2+b v_3\wedge v_4+c v_5\wedge v_6+d v_7\wedge (v_1+v_3),$$
for some $a,b,c,d \in \mathbb{C}$.
By simple calculation we have $w^2=0$ if and only if exactly one of the
$a,b,c,d$ is nonzero, which gives a contradiction with the existence of
$p_5$ as in the assumption.
\end{proof}

Now assume that $L$ has a $g^1_5$. It follows from \cite[(2.7)]{Muk2} that it is
given by five points on $L$ spanning a $\mathbb{P}^3$. By Lemma \ref{g15 on G(2,V)} these points
are contained in a section of $G_2$ with a $G(2,5)$. We conclude by \cite[lem. 3.3]{M3l3}
as five isolated points cannot be a linear section of a cubic scroll by a $\mathbb{P}^3$.
\end{proof}

\begin{prop}
 Every smooth polarized surface $(S,L)$ which appears as a complete linear section of
$\hat{G}_2$ is a $K3$ surface with a $g^1_5$.
\end{prop}
\begin{proof}
As $\hat{G}_2$ is a flat deformation of $G_2$ it's smooth complete linear
section of dimension 2 are K3 surfaces of genus 10.
Moreover each of these surfaces contains an elliptic curve of degree 5 which is
a section of the Fano fourfold $F$.
\end{proof}
Let us consider the converse.
 Let $(S,L)$ be polarized K3 surface of genus 10, such that $L$  admits a
$g^1_5$ induced by and elliptic curve $E$ and do not admit a $g^1_4$. By the theorem of Green and Lazarsfeld \cite{GL}
this is the case for instance when $L$ admits a $g^1_5$ and does not admit neither a $g^1_4$ nor a $g^2_7$. 

We have $E.L=5$ and $E^2=0$ hence $h^0(O(L)|_E)= 5$ and $h^0(O(L-E)|_E)= 5$.
It follows from the standard exact sequence that $h^0(O(L-E))\geq 6$ and
$h^0(O(L-2E))\geq 1$.
We claim that $|L-E|$ is base point free:
Indeed, denote by $D$ its moving part and $\Delta$ its fixed part. Clearly
$|D-E|$ is effective as $|L-2E|$ is.
Observe that $D$ cannot be of the form $kE'$ with $E'$ an elliptic curve,
because as $D-E$ is effective we would
have $E'=E$ hence $k\leq 3$ which would contradict $h^0(O(L-E))\geq 6$. Hence we
may assume
that $D$ is a smooth irreducible curve and $h^1(O(D))=0$. By Riemann-Roch we
have:
$$ 4+D^2=2 h^0(O(D))=2 h^0(O(D+\Delta)\geq 4+(D+\Delta)^2$$
and analogously:
$$ 4 =4+E^2=2 h^0(O(E))=2 h^0(O(E+\Delta)\geq 4+(E+\Delta)^2,$$
because $|D-E|$ being effective implies $\Delta$ is also the fixed part of
$|E+\Delta|$.
It follows that $L.\Delta=(D+E+\Delta).\Delta \leq 0$, which contradicts
ampleness of $L$.

It follows from the claim that $|L-E|$ is big, nef, base point free and
$h^0(O(L-E))= 6$.
Observe that $|L-E|$ is not hyper-elliptic. Indeed, first since $(L-E).E=5$ it
cannot be a double genus 2 curve. Assume now that
there exists an elliptic curve $E'$ such that $E'.(L-E)=2$ then $L.E\leq 4$
because $|L-2E|$ being effective implies $(L-2E).E'\geq 0$.
This contradicts the nonexistence of $g^1_4$ on $L$.

Hence $|L-E|$ defines a birational morphism to a surface of degree 8 in
$\mathbb{P}^5$. Observe moreover that the image of an element in $|L-2E|$
is a curve of degree 3 spanning a $\mathbb{P}^3$. The latter follows from the
fact that by the standard exact sequence
$$0\longrightarrow O(E)\longrightarrow O(L-E)\longrightarrow
O_{\Gamma}(L-E)\longrightarrow 0$$ for $\Gamma\in|L-2E|$ we have
$h^0((L-E)|_\Gamma)=4$.

Next we have two possibilities:
\begin{enumerate}
 \item The system $|L-E|$ is trigonal, then the image of $\varphi_{|L-E|}(S)$ is
contained in a cubic threefold scroll. The latter is either the
Segre embedding of $\mathbb{P}^1\times \mathbb{P}^2$ or a cone over a cubic
rational normal scroll surface.
\item The surface  $\varphi_{|L-E|}(S)$ is a complete intersection of three
quadrics.
\end{enumerate}

Moreover for the image $C=\varphi_{|L-E|}(\Gamma)$ we have the following
possibilities:
\begin{itemize}
 \item Either $C$ is a twisted cubic,
\item or $C$ is the union of a conic and a line,
\item or $C$ is the union of three lines.
\end{itemize}
Consider now the
composition $\psi$ of $\varphi_{|L-E|}$ with the birational
map given by quadrics in $\mathbb{P}^5$ containing $C$. It is given by
a subsystem of $|L|=|2(L-E)-(L-2E)|$. Moreover in every case above $\psi(S)$
spans a $\mathbb{P}^{10}$, because in each case the space of
quadrics containing $\varphi_{|L-E|}(S)$ is three dimensional. It follows that
$\psi$ is
defined on $S$ by the complete linear system $|L|$.
Finally $(S,L)$ is either a proper linear section of one of the three considered
degenerations of $G_2$ or a divisor in the blow up of a cubic scroll in a cubic curve.

In particular we have the following.
\begin{prop} \label{trig}
Let $(S,L)$ be a polarized K3 surface of genus 10 such that $L$ admits a $g^1_5$
induced by an elliptic curve $E$ but no $g^1_4$. If moreover $|L-E|$ is not trigonal,
then $(S,L)$ is a proper linear section of one of the four considered
degenerations of $G_2$.
\end{prop}
\begin{rem}\label{trigonal}
The system $|L-E|$ is trigonal on $S$ if and only if there exists an elliptic curve $E'$ on $S$ such that one of the following holds:
\begin{enumerate}
 \item $L.E'=6$ and $E.E'=3$
\item $L.E'=5$ and $E.E'=2$.
\end{enumerate}
Now observe that in both cases we obtain a second $g^1_5$ on $L$. In the first case it is given by the restriction of $E'$ and in the 
second we get at least a $g^2_7$ by restricting $|L-E-E'|$, the latter gives rise to a $g^1_5$ by composing the map with a projection 
from the singular point of the image by the $g^2_7$ (there is a singular point by Noether's genus formula).   
\end{rem}
We can now easily prove Proposition \ref{one g15}. 
\begin{proof}{Proof of Proposition \ref{one g15}}
Indeed the existence of exactly one $g^1_5$ excludes both the existence of a $g^2_7$ and of a $g^1_4$, hence the $g^1_5$ is induced
by an elliptic pencil $|E|$ on $S$. Moreover by Remark \ref{trigonal} we see that $|L-E|$ is then trigonal.   
\end{proof}

The Proposition \ref{Pic 2} follows directly from Proposition \ref{trig} and the fact that in the more degenerate case we clearly 
get a higher Picard number due to the decomposition of $C$.

\begin{rem} The K3 surfaces obtained as sections of considered varieties fit to the case
$g=10$, $c=3$, $D^2=0$, and scroll of type $(2,1,1,1,1)$ from \cite{JK} (Observe that there
 is a misprint in the table, because $H^0(L-2D)$ should be 1 in this case). The embedding in the scroll
 corresponds to the induced embedding in the projection of $\hat{G}_2$ from the distinguished plane.
\end{rem}
 \section{Acknowledgements}
I would like to thank K. Ranestad, J. Buczy\'nski and G.Kapustka for their help.
I acknowledge also the referee for useful comments.
The project was partially supported by SNF, No 200020-119437/1 and by MNSiW, N N201 414539.


\begin{thebibliography}{RS}

\bibitem{ott} H. Abo, G. Ottaviani, Ch. Peterson,
 \emph{Non-Defectivity of Grassmannians of planes}, arXiv:0901.2601v1  [math.AG], to appear in J. Algebraic Geom. 
\bibitem{BB} V.V. Batyrev, L. A. Borisov, \emph{On Calabi-Yau Complete Intersections in
Toric Varieties}, in âHigher Dimensional Complex Geometryâ, M. Andreatta and T. Peternell eds., 1996, 37-65.
\bibitem{BCKS} V. V. Batyrev, I. Ciocan-Fontanine, B. Kim, D. van
Straten, \emph{Conifold transitions and mirror symmetry for Calabi-Yau complete
intersections in Grassmannians}. Nuclear Phys. B 514 (1998), no. 3, 640--666.
\bibitem{BCKS2}V. V. Batyrev, I. Ciocan-Fontanine, B. Kim, D. van Straten,
\emph{Mirror symmetry and toric degenerations of partial flag manifolds}. Acta Math. 184 (2000), no. 1, 1--39.
\bibitem{BK}V. V. Batyrev, M. Kreuzer, \emph{Constructing new Calabi-Yau 3-folds and their mirrors via conifold transitions}, preprint 2008, arXiv:0802.3376v2.
\bibitem{Mag} W. Bosma, J. Cannon, C. Playoust, \emph{The Magma algebra system. I. The user language}, J. Symbolic Comput., 24 (1997), 235--265.
\bibitem{Br} M. Brion, V. Alexeev, \emph{Toric degenerations of spherical varieties},
Selecta Math. (N.S.) 10 (2004), 453--478.
\bibitem{M2} D. R. Grayson, M. E. Stillman:
\emph{Macaulay 2, a software system for research in algebraic geometry},
Available at  http://www.math.uiuc.edu/Macaulay2/
\bibitem{GL} M. Green, R. Lazarsfeld, \emph{Special divisors on curves on a K3 surface}, Invent. Math. 89 (1987), 357--370.
\bibitem{JK} T. Johnsen, A. L. Knutsen, \emph{$K3$ projective models in
scrolls}.
Lecture Notes in Mathematics, 1842. Springer-Verlag, Berlin, 2004. viii+164
\bibitem{GK} G. Kapustka, \emph{Primitive contractions of Calabi--Yau
threefolds II}, J. London Math. Soc.J. Lond. Math. Soc. (2) 79 (2009), no. 1,
259--271.
\bibitem{Unpr} M. Kapustka, \emph{Geometric transitions between
Calabi-Yau threefolds related to Kustin-Miller unprojections.},
arXiv:1005.5558v1  [math.AG].
\bibitem{M3l3} M. Kapustka, K. Ranestad, \emph{Vector bundles on Fano
varieties of genus 10}, preprint 2010, arXiv:1005.5528v1 [math.AG].

\bibitem{Muk1} S. Mukai, \emph{Curves, K3 surfaces and Fano 3-folds of
genus $\leq 10$}, in Algebraic Geometry and Commutative Algebra in Honor of
Masayoshi Nagata, pp. 357-377, Kinokuniya, Tokyo, 1988.
\bibitem{Muk2} S. Mukai, \emph{Curves and Grassmannians.}  Algebraic geometry
and
related topics (Inchon, 1992),  19--40, Conf. Proc. Lecture Notes Algebraic
Geom.,
I, Int. Press, Cambridge, MA, 1993.
\bibitem{MukBN} S. Mukai, \emph{Non-abelian Brill-Noether theory and Fano
3-folds} [translation of Sugaku 49 (1997), no. 1, 1--24].  Sugaku Expositions
14  (2001),  no. 2, 125--153.

\bibitem{RavSrin} G. V. Ravindra, V. Srinivas, \emph{The Grothendieck-Lefschetz
theorem for normal
projective varieties}, J. Algebraic Geom. 15 (2006), 563--590.


\end{thebibliography}
\end{document}